\documentclass[12pt,a4paper]{article}
\usepackage{amsfonts}
\usepackage{epsfig}
\usepackage{graphicx}
\usepackage{amsmath}
\usepackage{amssymb}

\newcommand{\R}{\mathbb{R}}
\newcommand{\N}{\mathbb{N}}
\newcommand{\PP}{\mathbb{P}}
\newcommand{\Z}{\mathbb{Z}}
\newcommand{\Per}{\mbox{Per}}
\newcommand{\nrm}[1]{\|#1\|}
\newcommand{\Mod}[1]{\vert{#1}\vert}
\newcommand{\intpart}[1]{\lfloor {#1} \rfloor}

\newcounter{main}

\newtheorem{theorem}{Theorem}[section]
\newtheorem{proposition}[theorem]{Proposition}
\newtheorem{lemma}[theorem]{Lemma}

\newtheorem{remark}{Remark}[section]

\newcommand{\blanksquare}{\,\,\,$\sqcup\!\!\!\!\sqcap$}
\newenvironment{proof}{{\flushleft {\bf Proof: }}}{\blanksquare}

\newcounter{example}
{{\stepcounter{example}}{\flushleft {\bf Example \arabic{example}:}}}%
{\par}

\title{\textbf{Abundance of elliptic dynamics on conservative $3$-flows}}

\author{ M\'{a}rio Bessa \thanks{Supported by FCT-FSE, SFRH/BPD/20890/2004. }
\space and Pedro Duarte }

\begin{document}

\maketitle

\begin{abstract}
We consider a compact $3$-dimensional boundaryless Riemannian manifold $M$ and the set of divergence-free (or zero divergence) vector fields without singularities, then we prove that this set has a $C^1$-residual (dense $G_{\delta}$) such that any vector field inside it is Anosov or else its elliptical orbits are dense in the manifold $M$. This is the flow-setting counterpart of Newhouse's Theorem 1.3~\cite{N}. Our result follows from two theorems, the first one is the $3$-dimensional continuous-time version of a theorem of Xia~\cite{X} and says that if $\Lambda$ is a hyperbolic invariant set for
some class $C^1$ zero divergence vector field $X$ on $M$, then either $X$ is Anosov, or else $\Lambda$ has empty interior. The second one is a version, for our $3$-dimensional class, of Theorem 2 of Saghin-Xia ~\cite{SX} and says that, if $X$ is not Anosov, then for any open set $U\subseteq M$ there exists $Y$ arbitrarily close to $X$ such that $Y^t$ has an elliptical closed orbit through $U$.
\end{abstract}

\emph{keywords:} Hyperbolic sets, Dominated splitting, Volume-preserving flows.\\

\section{Introduction and statement of the results}
In the beginning of the 1980s Ma\~{n}\'{e} proved (\cite{M4}) that $C^1$-robust tran\-si\-ti\-vi\-ty of diffeomorphisms in surfaces lead to hyperbolicity. One of the main ideas to prove this result was the construction of $C^0$-perturbations of $2$-dimensional linear systems (the dynamical cocycle) over the periodic points. He proved that under the presence of a weak form of hyperbolicity it is not possible to perturb a periodic point in order to obtain a matrix with one Lyapunov exponent zero (eigenvalue with real part equal to one). Ma\~{n}\'{e}'s result was extended for $3$-flows by Doering (\cite{D}), for 3-diffeomorphisms by D\'{i}az-Pujals-Ures (\cite{DPU}) obtaining partial hyperbolicity instead of hyperbolicity, for $n$-diffeomorphisms $n\geq{4}$ by Bonatti-D\'{i}az-Pujals (\cite{BDP}) obtaining an even weaker form of hyperbolicity (dominated splitting) and finally the $n$-flows version of this last result by Vivier (\cite{V}) and its conservative (or volume-preserving) counterpart (\cite{BR}). 

If we consider the $C^1$-generic dynamics of systems which preserves a given volume form, then the ergodic theoretic results in the vein of these previous ones are the proof of the Bochi-Ma\~{n}\'{e} Theorem (\cite{M2,B}), its generalization to higher dimensional diffeomorphisms (\cite{BV2}) and the $3$-flows counterpart of Bochi-Ma\~{n}\'{e}'s theorem (\cite{B2}). 

We recall that Arbieto-Matheus (\cite{AM}) proved that $C^1$-robust transitive volume-preserving $3$-flows must be Anosov with the help of a very usefull perturbation lemma for zero divergence (or divergence-free) vector fields and also that  Ali-Horita (\cite{HT}) proved that robustly transitive symplectomorphisms must be partially hyperbolic. 

All these results are supported, as we already said, in $C^1$-perturbations of the action of the tangent map (or flow) along orbits. We also find other examples of $C^1$-type results by going back to the seventies and recall a theorem of Newhouse (\cite{N}) fundamental in the generic theory of conservative diffeomorphism in surfaces. Newhouse's theorem says that $C^{1}$-generic area-preserving diffeomorphisms in surfaces are Anosov or else the elliptical points are dense. Note that the proof is strongly supported in the symplectic structure. Later, in the $4$-dimensional symplectomorphisms context, Arnauld presented in~\cite{A} a refined version of~\cite{N}. Recently, in~\cite{SX}, Saghin-Xia generalized Arnauld result for the multidimensional symplectic case. Our aim is to obtain similar results in the context of volume-preserving $3$-flows.

\medskip

We end this introduction describing our main results here.
We start by ge\-ne\-ra\-li\-zing Proposition 1.1 of~\cite{X} to the context of
$3$-dimensional $C^1$ zero divergence vector fields.

Unless we say something in contrary in  this paper $M$ will always denote a $3$-dimensional compact, connected, boundaryless  and smooth Riemannian manifold.

\begin{theorem}\label{intvazio}
Let $\Lambda$ be a hyperbolic invariant set for
some class $C^1$ zero divergence vector field $X$ on $M$.
Then either $X$ is Anosov on $M$, or else $\Lambda$
has empty interior.
\end{theorem}

Denote by $\mu$ the Lebesgue measure induced by the Riemannian volume form on $M$.
Let $\mathfrak{X}^{1}_{\mu}(M)$ denote the set of all class $C^{1}$
zero divergence  vector fields, i.e., those which preserve $\mu$. 
We assume that $\mathfrak{X}^{1}_{\mu}(M)$ is
endowed with the $C^{1}$ topology. 
This topology is defined by a norm, which we shall denoted by $\|\cdot\|_{C^1}$.
The subset of $\mathfrak{X}^{1}_{\mu}(M)$,
formed by those vector fields without singularities, will be
denoted here by $\mathfrak{X}^{1}_{\mu}(M)^{*}$.  
With this notation, we prove that:

\begin{theorem}\label{main}
Given $\epsilon>0$, $p\in M$,
$U\subseteq M$ an open set with $p\in U$, and 
$X\in \mathfrak{X}^{1}_{\mu}(M)^{*}$ a vector field
which is not  Anosov,
there is $Y\in \mathfrak{X}^{1}_{\mu}(M)^{*}$
such that  $\| X-Y\|_{C^1} < \epsilon$,\, and $Y$ has an elliptic
closed orbit through $U$.
\end{theorem}

In order to obtain the previous theorem, the main tool is to
adapt the perturbation techniques developed in~\cite{B1,B2} to the periodic context together with the Pasting Lemma in~\cite{AM}. From this result we are able to prove a 
dichotomy which is the continuous-time version of Newhouse's theorem in~\cite{N}.

\begin{theorem}\label{main_theor}
There exists a $C^1$-residual subset $\mathcal{R}\subset \mathfrak{X}^{1}_{\mu}(M)^{*}$
such that,
if $X\in{\mathcal{R}}$ then $X$ is Anosov or 
else the elliptical periodic points of $X^t$ are dense in $M$. 
\end{theorem}

\medskip

\section{Preliminaries}

We consider $C^{r}$, $r\geq 1$, class $C^1$ zero divergence vector fields 
$X:M\rightarrow{TM}$. 
Integrating any $C^1$ vector field $X$ we obtain its associated flow $X^{t}$, 
which is a $1$-parameter group of $C^1$ conservative diffeomorphisms. 
The infinitesimal generator of the flow $X^t$ is the vector field $X$ say, $\frac{dX^{t}}{dt}|_{t=s}(p)=X(X^{s}(p))$. The flow $X^{t}$ has a tangent map $DX^{t}_{p}$ which is the solution of the
non-autonomous linear differential equation
$\dot{u}(t)=DX_{X^{t}(p)}\cdot u(t)$  called the \emph{linear
variational equation}. 
Given a vector field  $X\in\mathfrak{X}^{1}_{\mu}(M)$, 
we say that $x\in M$ is a \emph{regular point}
if  $X(x)\not={0}$. Otherwise, if $X(x)=0$ then $x$ is called
a \emph{singularity} of $X$. 
We shall denote by $R$ the set of all regular points of $X$. 
Let $\mathbb{R}X(x)$ be the direction of the vector field. 
We define the normal sub-bundle at $x\in{R}$ by $N_{x}=\mathbb{R}X(x)^{\perp}$. 
Let $\Pi_{X^{t}(x)}$ be the orthogonal projection onto
$N_{X^{t}(x)}$. 
The map 
$$ \begin{array}{cccc}
P^{t}_{X}(x): & N_{x} & \longrightarrow & N_{X^{t}(x)} \\
& v & \longmapsto & \Pi_{X^{t}(x)}\circ DX^{t}(x)\cdot{v},
\end{array}  $$
is called the \emph{linear Poincar\'{e} flow} of $X$ at $x$.
Like the  tangent map $DX^{t}_{p}$, the
 linear Poincar\'{e} flow $P_{X}^{t}(x)$
is solution of the following variational equation:
$$\dot{u}(t)=\Pi_{X^{t}(x)}\circ DX_{X^{t}(p)}\cdot u(t).$$

Given a linear map $A$ let $$\|A\|:=\underset{v\not=\vec{0}}{\sup}\frac{\|A(v)\|}{\|v\|}$$ denote the norm of $A$.

We say that a compact  $X^{t}$-invariant set $\Lambda\subset{M}$,
without singularities,  is \emph{uniformly hyperbolic} for the linear Poincar\'{e} flow $P_{X}^{t}(x)$
\, if there is, over $\Lambda$, a $P_{X}^t$-invariant continuous splitting $N= N^{u}\oplus{N^{s}}$,
and there are constants $C>0$ and $\gamma\in(0,1)$ such that for all $x\in{\Lambda}$ we have:
$$\|P_{X}^{-t}(x)|_{{N}^{u}_{x}}\|\leq{C\gamma^{t}}\,,\;\text{   and also  }\; \|P_{X}^{t}(x)|_{{N}^{s}_{x}}\|\leq{C\gamma^{t}}\text{  for all }t\geq 0.$$

\begin{remark}\label{Doering}
It follows by Proposition 1.1 of~\cite{D} that the hyperbolicity for the linear Poincar\'e flow, 
over a compact set, is equivalent to the hyperbolicity of the tangent flow over the same set.
\end{remark}

Next we recall a weaker form of uniform hyperbolicity which will be very useful to prove our results. 
We say that a compact $X^{t}$-invariant set $\Lambda\subset{M}$ 
has \emph{dominated splitting}, or \emph{uniformly projective hyperbolicity}, 
for the linear Poincar\'e flow of the vector field $X$ 
if there is, over $\Lambda$, a $P_{X}^t$-invariant continuous splitting 
$N=N^{u}\oplus{N^{s}}$, and there are constants $C>0$ and $\gamma\in(0,1)$ such that for all $x\in{\Lambda}$ we have:
\begin{equation}\label{ds}
\frac{\|P_{X}^{t}(x)|_{{N}^{s}_{x}}\|}{\|P_{X}^{t}(x)|_{{N}^{u}_{x}}\|}\leq{C\gamma^{t}}\text{  for all }t\geq 0.
\end{equation}
In the presence of a dominated splitting both sub-bundles may expand.
In such case,  $N^u$ expands more than $N^s$. 
It may also happen that both sub-bundles contract, in which case  $N^u$ contracts less than $N^s$. 
For a complete description of dominated splitting structures see~\cite{BDV} section B.1. 
The following simple lemma give us an equivalent characterization of uniformly projective hyperbolicity. 
We leave the proof to the reader.
\begin{lemma}\label{equivalencia}
Let $\Lambda\subset{M}$ be a compact  $X^{t}$-invariant set.
Then\, $\Lambda$ is uniformly projectively hyperbolic \, if and only if \,
there is $m\in{\mathbb{N}}$ such that 
\begin{equation}\label{equiv}
\frac{\|P_{X}^{m}(x)|_{{N}^{s}_{x}}\|}{\|P_{X}^{m}(x)|_{{N}^{u}_{x}}\|}\leq{1/2}\text{ for all }\Lambda. 
\end{equation}
\end{lemma}

When the condition (\ref{equiv}) is satisfied we shall also say that 
$\Lambda$  has $m$-\emph{dominated splitting}, and we will refer to $m$ as
the \emph{domination time}.
Lemma~\ref{equivalencia} says that we may either consider the pair of constants 
 $(C,\gamma)$, or else the pair $(m,\frac{1}{2})$,
to characterize uniformly projective hyperbolicity. 

We denote by $\Per(X^{t})$ the set of all periodic points of $X^{t}$,
and by $Per_{\tau}(X^{t})$ the subset of periodic points 
with period less or equal than $\tau$. 
Given $p\in \Per(X^{t})$, the real number  
$$\lambda_{u}(p):=\underset{t\rightarrow \pm \infty}{\lim} \frac{1}{t}\log\|P_{X}^{t}(p)\|$$
is called the \emph{upper Lyapunov exponent} of the orbit $X^{t}(p)$. 
If this number is zero, then, by conservativeness, $p$ is an elliptical periodic point. 
Otherwise, the real eigenvalues $\sigma^{\pm 1}$ of the map $P_{X}^{\tau}(p)$, 
where $\tau$ is the period of $p$, satisfy  
$$e^{\lambda_{u}(p)\tau}=|\sigma|>1>|\sigma^{-1}|=e^{-\lambda_{u}(p)\tau}.$$

\noindent Let $\Per_{\text{hyp}}(X^{t})$ denote the subset of all hyperbolic periodic points
in $\Per(X^{t})$.
Note that any given periodic orbit in  $\Per_{\text{hyp}}(X^{t})$
has dominated splitting, but the domination time $m$ may be very large.
Hence, to \emph{see} the hyperbolic behavior of an orbit in $\Per_{\text{hyp}}(X^{t})$
we may have to consider  large iterates. 
The function $m:\Per_{\text{hyp}}(X^{t})\rightarrow{\mathbb{N}}$ is 
in general unbounded. 
Also, the weak hyperbolic behavior relates with the splitting angle being 
close to zero. 
Clearly, given a uniformly hyperbolic invariant set
$\Lambda\subset \overline{\Per_{\text{hyp}}(X^{t})}$,  
the splitting angle, between $N^u$ and $N^s$,
 is bounded away from zero over $\Lambda$. 
This tranversality follows from the continuity of the splitting $N =N^u\oplus N^s$, 
and the compactness of $M$.
Given a vector field $X\in\mathfrak{X}^{1}_{\mu}(M)$, we define
$$\Delta_{m}(X)=\left\{x\in \Per_{\text{hyp}}(X^t): \; \frac{\|P_{X}^{m}(x)|_{{N}^{s}_{x}}\|}{\|P_{X}^{m}(x)|_{{N}^{u}_{x}}\|}\geq{\frac{1}{2}}\right\}\;,$$ 
and
$$\Lambda_{m}(X)=\left\{x\in \Per_{\text{hyp}}(X^t): \; \frac{\|P_{X}^{m}(X^t(x))|_{{N}^{s}_{x}}\|}{\|P_{X}^{m}(X^t(x))|_{{N}^{u}_{x}}\|}\leq{\frac{1}{2}}\,,\;
\forall\, t\geq 0\; \right\}\;.$$ 
By definition, the closure of $\Lambda_{m}(X)$ has $m$-dominated splitting,
and, by conservativeness, nonexistence of singularities and the $3$-dimensionality, it is a hyperbolic invariant set. 
Of course the set $\Per(X^{t})$ decomposes as the disjoint union
of $\Lambda_{m}(X)$ with the  superset of $\Delta_{m}(X)$ saturated by the flow:\,
$\underset{t\in{\mathbb{R}}}{\cup}X^{t}(\Delta_{m}(X))$.

We recall the definition of \emph{flowbox} along a regular orbit segment.
Take $p\in \Per(X^{t})$, of period $\geq\tau$, and the standard Poincar\'{e} map based at $p$, i.e. the map, $\mathcal{P}_{X}^{t}(p)\colon\mathcal{V}_{p}\subseteq \mathcal{N}_{p}\rightarrow{\mathcal{N}_{X^{t}(p)}}$, where $\mathcal{N}_{p}$ is a surface with tangent plane $N_{p}$ and $\mathcal{V}_{p}$ is a small enough neighborhood of $p$. 
Using the implicit function theorem, we get a function 
describing the first-return-time  
of orbit trajectories from $\mathcal{V}_{p}$ to $\mathcal{N}_{p}$.
Denote it by $a:\mathcal{V}_{p}\rightarrow \mathbb{R}$. 
Given $B\subseteq\mathcal{V}_{p}$, the set $$\mathcal{F}_{X}^{\tau}(p)(B):=
\{\mathcal{P}_{X}^{t}(p)(q): \, q\in{B}, \, t=t(q)\in[0,a(q)[\, \},$$ 
is called the \emph{time-$\tau$ length flowbox} at $p$ associated to the vector field $X$.

\medskip

We end this section giving some notation for conservative linear differential systems.
Any function $A:\R\to sl(2,\R)$ of class $C^1$,
i.e., a traceless matrix function, defines a linear variational equation
\begin{equation}
\label{lve}
 \dot{u}(t)= A(t)\cdot u(t)\quad (t\in\R)\;.
\end{equation}
We denote by $\Phi_A(t) \in SL(2,\R)$
the fundamental solution of~(\ref{lve}).

\medskip

\section{Some useful perturbation lemmas}

To prove our results in the conservative setting it will be necessary to perform the perturbations in this context, so the following two  lemmas will be central:

\begin{lemma}(Conservative flowbox theorem)\label{cfbt}
Given a vector field $X\in\mathfrak{X}_{\mu}^{s}(M)$, $s\geq{2}$, and 
$\tau>0$ 
there is some $C=C(\tau)>0$ such that for any
arc of regular orbit 
$\Gamma_{X}^{\tau}(x):=\{X^t(x): 0\leq t\leq \tau\}$
of length $\tau$,
there exists a $C^s$-conservative change of coordinates $\Psi$ from $X$ into $T=(1,0,0)$, 
defined over a neighborhood $V$ of $\Gamma_{X}^{\tau}(x)$,
such that $\|\Psi\|_{C^s}\leq C$ on $V$.
\end{lemma}
\begin{proof}
 See~\cite{B2} Lemma 3.4.
\end{proof}

\begin{lemma}\label{AM}(Arbieto-Matheus Pasting Lemma) Given $\epsilon>0$ there exists $\delta>0$ such that if $X\in\mathfrak{X}^{\infty}_{\mu}(M)$, $K\subset M$ is a compact set and $Y\in\mathfrak{X}^{\infty}_{\mu}(M)$ is $\delta$-$C^{1}$-close to $X$ in a small neighborhood $U\supset K$, then there exist $Z\in\mathfrak{X}^{\infty}_{\mu}(M)$, $V$ and $W$ with $K\subset V\subset U\subset W$ such that:
\begin{description}
 \item 1) $Z|_{V}=Y$;
\item  2) $Z|_{int(W^c)}=X$;
\item  3) $Z$ is $\epsilon$-$C^1$-close to $X$.
\end{description}
\end{lemma}
\begin{proof}
See~\cite{AM} Theorem 3.1. 
\end{proof}

\medskip

Given any divergence-free vector field $X$, any regular point $p\in M$ and $t\in\mathbb{R}$ it is clear that the linear Poincar\'e flow $P_{X}^{t}(p):N_{p}\rightarrow N_{X^{t}(p)}$ satisfy the equality:
\begin{equation}\label{volume}
|\det P_{X}^{t}(p)| \cdot \|X(X^{t}(p))\|=\|X(p)\|. 
\end{equation}
Thus the linear differential system $P_{X}^{t}$ based in $X^{t}$ is \emph{modified area-preserving},
 cf.~\cite{B1} Section 2.1.
Let $p\in M$ be a regular point and \,
$\Gamma_{X}^{\tau}(p)$ \,
be an injective orbit segment, where $\tau$ is less than the period of $p$
in case $p$ is periodic. 

Define $\gamma:[0,\tau]\to M$ to be the flow parameterization of $\Gamma_{X}^{\tau}(p)$,
$\gamma(t)=X^t(p)$.
We denote by $N_R$ the normal bundle to the trajectories
of the flow $X^t$, over the invariant set $R$ of all regular points.
Let $N_\gamma = \gamma^{*} N_R$ be the pull-back by $\gamma$,
which has fiber $N_{\gamma(t)}$ at each base point $t\in [0,\tau]$.
The linear Poincar\'{e} flow $P_X^t$ on the normal bundle $N_R$
induces the flow of linear maps
$P_\gamma^t(s) = P_X^t(\gamma(s))$,
on the restricted normal bundle  $N_\gamma$.
From~(\ref{volume}), we get the following conservativeness
relation for this restricted linear flow
\begin{equation}\label{vol2}
|\det P_{\gamma}^{t}(s)| \cdot \|\gamma'(s+t)\|=\|\gamma'(s)\|. 
\end{equation}

Now, the linear bundle $N_\gamma$ is trivial. 
We can define an explicit
linear bundle isomorphism 
$$\Psi:[0,\tau]\times \R^2\to N_\gamma,\quad\text{
such that}\quad
\Psi(t,x,y)=(t,\psi_t(x,y))\,, $$
where for each $t\in[0,\tau]$,
$\psi_t:\R^2\to N_\gamma(t)$ is a linear map
with determinant, i.e., volume dilatation coefficient, equal to $1/\nrm{\gamma'(t)}$.
Such family of linear maps $\psi_t$ is not unique.
To define one take any non vanishing section $\eta(t)$ of $N_\gamma$,
and normalize it so that $\nrm{\eta(t)}=1/\nrm{\gamma'(t)}$ for all $t\in [0,\tau]$.
Then $\psi_t:\R^2\to N_\gamma(t)$ may be defined by
$$ \psi_t(1,0)=\eta(t)\quad\text{ and }\quad \psi_t(0,1)=\eta(t)\times \gamma'(t)\;.$$
Using the isomorphism $\Psi$ we can conjugate the restricted linear flow
$P_\gamma^t$ on $N_\gamma$ to a linear flow $L_\gamma^t$ on the constant bundle
$[0,\tau]\times\R^2$.
Defining the linear maps,
$L^t(s,u,v) = \left( \,s+t,\, M_t(u,v) \, \right)$, with
$$M_t(u,v)=(\psi_{s+t})^{-1}\circ P_\gamma^t (s) \circ\psi_s(u,v)
\;,$$
it is clear that they form a linear flow, in the sense that
whenever $s$, $s+r$ and $s+r+t$ are in $[0,\tau]$,\,
$L_\gamma^{t+r}(s) = L_\gamma^{t}(s+r) \circ L_\gamma^{r}(s)$.
From~(\ref{vol2}), it follows that all linear maps $M_t$ have
determinant one. Therefore, 

\begin{lemma}\label{flowL}
The linear maps $L^{t}$ form a volume-preserving flow. 
\end{lemma}

\medskip

In the next lemma we consider a change of coordinates 
useful to perform volume-preserving perturbations of the linear Poincar\'{e} flow along a periodic orbit. 
We shall denote by $L$ the 
zero divergence vector field associated to the linear flow $L^{t}$.

\medskip

\begin{lemma}\label{trivial}
Given a vector field $X\in\mathfrak{X}_{\mu}^{s}(M)$, $s\geq{2}$, and $\tau>0$
there is some $C=C(\tau)>0$ such that for any
arc of regular orbit 
$\Gamma_{X}^{\tau}(x)$
of length $\tau$,
there exists a conservative $C^{s}$
diffeomorphism $\Psi$,
defined over a neighborhood of $\Gamma_{X}^{\tau}(x)$,
such that $L=\Psi_{*}X$ and $\|\Psi\|_{C^s}\leq C$. 
\end{lemma}
\begin{proof}
By Lemma~\ref{cfbt} there exists a conservative $C^{1}$ diffeomorphism $\Psi_{1}$ such that $T=(\Psi_{1})_{*}X$, where $T$ is a trivial vector field $T=(1,0,0)$. Now applying the same result we find a conservative $C^{s}$ diffeomorphism $\Psi_{2}$ such that $T=(\Psi_{1})_{*}L$. Finally, we define  $\Psi:=\Psi_{2}^{-1}\circ\Psi_{1}$. Hence $\det D\Psi=\det D\Psi_{2}^{-1}.\det D\Psi_{1}=1$.
\end{proof}

\medskip

In this lemma, the constant $C(\tau)$ grows exponentially with $\tau$,
but we shall always use it with $\tau=1$. In order to make
a lengthy perturbation along some large periodic orbit, instead
of doing a single (and costly) perturbation, we split the period
into many time-one  disjoint intervals and then perform several
disjointly supported (local) perturbations. The global perturbation size
is small, because it is just the maximum size of all local perturbations.
All perturbations will be carried out in the linearizing co-ordinates
provided by Lemma~\ref{trivial},
after which lemma~\ref{AM} is used to extend conservatively 
the linear vector field into a zero divergence vector field that
coincides with the original vector field outside a small neighbourhood
of the periodic orbit. 

\medskip

Let $R_\xi$ denote the $\xi$-angle  rotation in any Euclidean $2$-plane.

\begin{lemma}\label{rot}
Given a hyperbolic matrix $A\in\mbox{SL}(2,\R)$, let
$\theta = \angle(E^s,E^u)$ be the angle between matrix $A$ eigen-directions.
Assume the rotation $R_\theta$ of angle $\theta$ takes the unstable direction
onto the stable direction of $A$, i.e., $R_\theta \, E^u = E^s$.
Then the matrix $A\, R_\theta$ is elliptic.
\end{lemma}

\begin{proof}
Take a hyperbolic matrix $A\in\mbox{SL}(2,\R)$. Let
$\theta = \angle(E^s,E^u)$ be the angle between the eigen-directions
of matrix $A$,
and set $B = A\,R_\theta$.
Consider the action of matrices $A$ and $B$ on the projective line $\PP^1 = \R/\pi\,\Z$,
described by the diffeomorphisms
$f_A:\PP^1\to\PP^1$ and $f_B:\PP^1\to\PP^1$.
Lift these maps to diffeomorphisms
$\tilde{f}_A:\R\to\R$ and $\tilde{f}_B:\R\to\R$
such that  
$\tilde{f}_A(x+\pi)= \tilde{f}_A(x)+\pi$
and 
$\tilde{f}_B(x+\pi)= \tilde{f}_B(x)+\pi$,
for all $x\in \R$. Notice that both $\tilde{f}_A$ and $\tilde{f}_B$ 
are increasing functions, since $A$ and $B$ have determinant one.
The definition of $\theta$ shows that the lifting $\tilde{f}_B$  can be chosen to
satisfy the relation
$\tilde{f}_B(x) = \tilde{f}_A(x+\theta)$, for all $x\in\R$.
Since $A$ is hyperbolic, $f_A$ has two fixed points:
an expanding fixed point $x^u$, and a contractive fixed point
$x^s$. We can choose the lifting $\tilde{f}_A$ so that
it has two families of fixed points, 
$x^u+k\,\pi$ and $x^s+k\,\pi$, with $k\in\Z$, and we may 
assume the fixed points $x^s,\,x^u\in\R$ satisfy
$\Mod{x^s-x^u}=\theta$.
In order to prove that $B$ is elliptic it is enough to
show that $f_B$ has non zero rotation number,
which amounts to say that $\tilde{f}_B(x)-x$
keeps a constant sign as $x$ runs through $\R$.
Two cases may occur: $x^s< x^u$ and $x^u< x^s$.
Assume first that $x^s< x^u$.
Then $-\theta <\tilde{f}_A(x)-x<0$ for all $x\in ] x^s, x^u[ $,
and $\tilde{f}_A(x)-x>0$ for all $x\in ] x^u, x^s+\pi[ $.
This implies that
$\tilde{f}_A(x)-x> -\theta$, for all $x\in\R$.
Therefore,
$\tilde{f}_B(x)-x =\tilde{f}_A(x+\theta)-(x+\theta) +\theta>
-\theta +\theta = 0$, for every $x\in\R$,
proving that $B$ is elliptic.
Assume now that $x^u< x^s$.
In this case $0 <\tilde{f}_A(x)-x< \theta $ for all $x\in ] x^u, x^s[ $,
and $\tilde{f}_A(x)-x<0$ for every $x\in ] x^s, x^u+\pi[ $.
But this implies that
$\tilde{f}_A(x)-x< \theta$, for all $x\in\R$.
Thus,
$\tilde{f}_B(x)-x =\tilde{f}_A(x+\theta)-(x+\theta) +\theta<
-\theta +\theta = 0$, for every $x\in\R$,
which proves that $B$ is elliptic.
\end{proof}

\medskip

%Given a matrix $S\in SL(2,\R)$,
%we call conservative linear $\delta$-isotopy from the identity to $S$
%to any smooth function $R:\R\to SL(2,\R)$  such that:
%\begin{enumerate}
%\item $R(t)= Id$\; for $t\leq 0$,
%\item $R(t)= S$ \; for $t\geq 1$, 
%\item $\| R'(t)\cdot R(t)^{-1}\| \leq \delta$\; for all $t\in\R$.
%\end{enumerate}

\begin{lemma}\label{abstr:perturb}
Given $C>0$ and $\epsilon>0$ there is some $\delta>0$ such that
for any smooth function $A:\R\to sl(2,\R)$ with $\|A(t)\|\leq C$ for $t\in\R$
and a matrix $S\in SL(2,\R)$ with $\|S-Id\|\leq \delta$
 there exists a smooth function $B:\R\to sl(2,\R)$
such that:
\begin{description}
\item 1) $\|A(t)-B(t)\|\leq \epsilon$\; for all $t\in\R$,
\item 2) $A(t)=B(t)$ \; for all $t\notin ]0,1[$, and
\item 3) $\Phi_B(1)= \Phi_A(1)\cdot S$. 
\end{description}
\end{lemma}

\begin{proof}
Taking $\delta>0$ small enough we may assume that
$S=e^Q$ for some traceless matrix $Q\in sl(2,\R)$.
Consider now a bump-function $\alpha:\mathbb{R}\rightarrow [0,1]$ such that $\alpha(t)=0$ for all $t\leq 0$, $\alpha(t)=1$ for all $t\geq 1$ and $\alpha^{\prime}(t)<2$ for all $t\in\mathbb{R}$. We define $\Phi:\R\to SL(2,\R)$,
$\Phi(t)= \Phi_{A}(t)\cdot e^{\alpha(t)\,Q}$.
By taking derivatives with respect to $t$ we obtain:
\begin{eqnarray*}
\Phi'(t) &=& (\Phi_A(t)\cdot e^{\alpha(t)\, Q})^{\prime} =\\
&=& \Phi_A^{\prime}(t)\cdot e^{\alpha(t)\, Q}+
\alpha'(t)\,\Phi_A(t)\cdot Q \cdot e^{\alpha(t)\, Q} =\\
&=& A(t)\cdot \Phi_A(t)\cdot e^{\alpha(t)\, Q}\; +\\
& & \;+\;  \alpha'(t)\,\Phi_A(t)
\cdot Q
\cdot \Phi_A(t)^{-1}
\cdot \Phi_A(t)
\cdot e^{\alpha(t)\, Q} =\\
&=& [A(t)+H(t)]\cdot \Phi_A(t)\cdot e^{\alpha(t)\, Q}\\
&=& [A(t)+H(t)]\cdot \Phi(t),
\end{eqnarray*}
where $H(t)=\alpha'(t)\,\Phi_A(t)
\cdot Q
\cdot \Phi_A(t)^{-1}$. 
Note that $Tr(H(t))=Tr(Q)=0$. Therefore, 
we have a map
$B:\R\to sl(2,\R)$ defined by
$B(t)= A(t)+H(t)$,
such that $\Phi(t)=\Phi_B(t)$  for all $t\in\R$.
In particular, for $t=1$ we have
$\Phi_B(1)= \Phi_{A}(1)\cdot e^{\alpha(1)\,Q}= \Phi_{A}(1)\cdot S$.
Because $\alpha'(t)$ is supported in $[0,1]$
we get property 2). Finally, take $\delta>0$ small enough so that
for $\| S-Id\|\leq \delta$ the unique matrix $Q\in sl(2,\R)$ such that
$S=e^Q$ satisfies $\|Q\|\leq e^{-2\, C}\,\epsilon/2$.
Then
\begin{eqnarray*}
\|B(t)-A(t)\| &=& \| H(t)\| \\
&\leq & |\alpha'(t)|\,\|\Phi_A(t)\|\,\| Q\|\, \| \Phi_A(t)^{-1}\| \\
&\leq & 2\,e^{2\,C}\,\| Q\|\leq \epsilon\;. 
\end{eqnarray*}
\end{proof}

\begin{remark}\label{smallangle1}
In order to obtain a similar conclusion like in lemma~\ref{abstr:perturb} but this time such that in 3) we have $\Phi_B(1)= S\cdot \Phi_A(1)$  we just switch $\Phi_{A}(t)\cdot e^{\alpha(t)\,Q}$ by $e^{\alpha(t)\,Q}\cdot\Phi_{A}(t)$ and proceed analogously.
\end{remark}

\medskip

%Take $X\in\mathfrak{X}_{\mu}^{s}(M)$ and $\epsilon>0$, 
%and define $J=J(X,\epsilon)$ such that:
%\begin{equation}\label{J}
%J=\max{\{\, \|P^{t}_{Z}(z)\|\, :\, z\in{M}, \,
%t\in[-1,1],\,
%Z\in\mathfrak{X}_{\mu}^{1} \; \text{s. t.}\,
%\|X-Z\|\leq\epsilon\,  \}}.
%\end{equation}

\begin{lemma}\label{smallang}
Given $\epsilon>0$ and $X\in\mathfrak{X}_{\mu}^{\infty}(M)$ there exists $\delta>0$
 such that given $\xi \in ]-\delta,\delta[$ and given
a periodic point $p\in M$ with period $\tau>1$,
there is some perturbation $Y\in\mathfrak{X}_{\mu}^{1}(M)$ of $X$ such that:
\begin{description}
 \item i) $\|X-Y\|_{C^1}<\epsilon$;
 \item ii) $X=Y$ outside $\mathcal{F}^{\tau}_{X}(p)(B(p,r))$;
 \item iii) $P_{Y}^{\tau}(p)=P_{X}^{\tau}(p)\circ R_{\xi}$.
\end{description}
\end{lemma} 

\begin{proof}
We first use Lemma~\ref{trivial} and assume that our flow is the linear one 
given by Lemma \ref{flowL}. 
Consider the linear variational equation associated to the linear Poincar\'e flow,
and let $A:\R\to sl(2,\R)$ be the equation coefficient matrix. Then
$\| A(t)\| \leq C$ for some constant $C>0$, independent of the 
chosen periodic orbit. Take $\delta>0$ given by the previous lemma,
and assume that $\xi\in ]-\delta,\delta[$.
Consider the rotation
$$R_{\xi}:=\begin{pmatrix}
\cos(\xi) & -\sin(\xi) \\
\sin(\xi) & \cos(\xi)   
  \end{pmatrix}.
$$ 
which satisfies $\|R_\xi-Id\|\leq |\xi|\leq\delta$.
By lemma~\ref{abstr:perturb},
there is a smooth function $B:\R\to sl(2,\R)$
satisfying 1), 2) and 
3): $\Phi_B(1)= \Phi_A(1)\cdot R_\xi$.
%Take any arbitrary small $r>0$ and $B(p,r)\subset N_{p}$. 
By Lemma~\ref{AM} we can extend conservatively the linear vector field induced by $B$ 
obtaining the required divergence-free vector field $Y$ in the conditions of the lemma. 
\end{proof}

\begin{remark}\label{smallangle}
By Remark~\ref{smallangle1} the same conclusion of Lemma \ref{smallang} holds with $P_{Y}^{\tau}(p)=R_{\xi}\circ P_{X}^{\tau}(p)$ in the place of iii).
\end{remark}

\medskip

\begin{proposition}\label{propA}(Small angle perturbation)
Given $\epsilon>0$ and $X\in\mathfrak{X}_{\mu}^{1}(M)$, there is some\, $\theta>0$ \,
 such that given any  periodic hyperbolic point $p\in M$ with period $\tau>1$, and such that $\measuredangle(N_{q}^{u},N_{q}^{s})<\theta$ for some $q=X^{t}(p)$, $0\leq t\leq \tau$, 
then there is some perturbation $Y\in\mathfrak{X}_{\mu}^{1}(M)$ of $X$ such that\,
 $\|X-Y\|_{C^1}<\epsilon$\, and 
$p$ is an elliptical periodic point for $Y^{t}$ with period $\tau$.
\end{proposition}

\begin{proof}
We may assume that $X\in \mathfrak{X}_{\mu}^{\infty}(M)$ by some preliminary perturbation using Zuppa's Theorem (\cite{Z}). This proposition follows from Lemmas~\ref{rot} and~\ref{smallang}.
\end{proof}

\medskip

Now we consider a simple lemma which relates the usual norm and the norm of the maximum. Given $x,y\in M$ and $T:\mathbb{R}^{2}_{x}\rightarrow \mathbb{R}^{2}_{y}$ any linear map represented by the matrix
\begin{equation}\label{maximum}
T=\begin{pmatrix}
a_{11}  &  a_{12} \\
a_{21}  &  a_{22} \\
\end{pmatrix}, 
\end{equation}
with respect to the invariant $1$-dimensional splitting $E^{1}_{x}\oplus E^{2}_{x}$ and $E^{1}_{y}\oplus E^{2}_{y}$. We define the norm of the maximum by $\|T\|_{*}=\text{max}\{|a_{11}|,|a_{12}|,|a_{21}|,|a_{22}|\}$.
\begin{lemma}\label{max}
Given any linear map $T$ as above, we suppose that $\measuredangle(E^{1}_{\sigma},E^{2}_{\sigma}
)>\theta$ for $\sigma=x,y$, then $T$ satisfies:
\begin{description}
 \item 1) $\|T\|\leq{4\sin^{-1}(\theta)\|T\|_{*}}$.
\item  2) $\|T\|_{*}\leq{\sin^{-1}(\theta)\|T\|}$.
\end{description}
\end{lemma}
\begin{proof}
See~\cite{BV2} Lemma 4.5. 
\end{proof}

\begin{lemma}\label{interchange}
Given $X\in\mathfrak{X}_{\mu}^{\infty}(M)$, $\epsilon>0$ and $\theta>0$. There exists $m(\epsilon,\theta)\in\mathbb{N}$ such that if $p\in M$ is a 
periodic hyperbolic point, with period $\tau>m$, such that:
\begin{description}
 \item a) $\measuredangle(N_{x_{t}}^{u},N_{x_{t}}^{s})>\theta$ for all $x_{t}:=X^{t}(p)$, $0\leq t\leq \tau$ and
 \item b) for some $x_{t}=X^{t}(p)$, with $0\leq t\leq \tau$, we have $x_{t}\in\Delta_{m}(X)$,
\end{description}
then there exists a conservative $C^{1}$ perturbation $Y$ of $X$,
and a point $q=X^s(p)$ in the periodic orbit through $p$, such that:
\begin{description}
 \item i) $\|X-Y\|_{C^1}<\epsilon$;
 \item ii) $X=Y$ outside a small neighbourhood of
$\{ X^t(q): 0\leq t \leq m\}$;
 \item iii) $P_{Y}^{m}(N_{q}^{u})=N_{X^{m}(q)}^{s}$.
\end{description}
\end{lemma}
\begin{proof}
The proof is essentially the same of Lemma 3.15 of~\cite{B2}, nevertheless we outline it where. Using a) we can consider that the hyperbolic splitting along the orbit is orthogonal, by just using an area-preserving change of coordinates with norm bounded by $\sin^{-1}(\theta)$. We can choose a constant $c=c(X,\epsilon,\theta)>0$ such that if 
\begin{equation}\label{two conditions}
\exists r,t\in\mathbb{R} \colon 0\leq r+t\leq m\text { such that }\frac{\|P_{X}^{r}(x_t)|_{{N}^{s}_{x_t}}\|}{\|P_{X}^{r}(x_t)|_{{N}^{u}_{x_t}}\|}> c
\end{equation}
then, there is a direction $v\in {N}_{x_t}\setminus\{0\}$ with 
$$\measuredangle(v,N^{s}(x_{t}))<\delta \text { and }\measuredangle(P_{X}^{r}(x_{t})\cdot v,N^{u}(x_{t+r}))<\delta,$$
where $\delta=\delta(X,\epsilon)$ is given by Lemma~\ref{smallang}.
In Case (\ref{two conditions}) we proceed with two $\epsilon$-small perturbations: one using Lemma~\ref{smallang} which sends $N^{u}_{x_{t}}$ onto $v\cdot \mathbb{R}$, another following Remark~\ref{smallangle} which sends $P_{X}^{r}(x_{t}) (v\cdot \mathbb{R})$ onto $N^{s}_{x_{t+r}}$, and we are over. 

Otherwise, if (\ref{two conditions}) is false, then there exists $d=d(c)\in\mathbb{R}$ such that for all $r\in[0,m]$ we have:
\begin{equation}\label{d}
d^{-1}\leq\frac{\|P_{X}^{r}(p)|_{{N}^{s}_{p}}\|}{\|P_{X}^{r}(p)|_{{N}^{u}_{p}}\|}\leq d. 
\end{equation}
We take $m=m(\epsilon,\theta)$ sufficiently large compared with $d$. Then we can combine $m$ small rotations with disjoint supports and $\epsilon$-small, obtained through Lemma~\ref{smallang}, to achieve the desired effect. Notice that the fact that the perturbations were disjointly supported guarantees that the size of the combined perturbation is the maximum and not the sum of all the $m$ perturbations. 
\end{proof}

\medskip

\begin{lemma}\label{mainlemma}
Let $X\in\mathfrak{X}_{\mu}^{\infty}(M)$, $\epsilon>0$ and $\theta>0$ be given. 
Let $m=m(\epsilon,\theta)\in\mathbb{N}$ be given by Lemma~\ref{interchange}. 
Then there exist $K=K(\theta,m)\in\mathbb{R}$ such that 
given any hyperbolic periodic orbit $\gamma$ with period $\tau>m$ 
satisfying a) and b) of Lemma~\ref{interchange},
then there exists a conservative $C^{1}$-perturbation $Y$ of $X$ such that:
\begin{enumerate}
\item $\| X-Y\|_{C^1}<\epsilon$;
\item  $\|P_{Y}^{\tau}(q)\|<K$, for some $q\in\gamma$.
\end{enumerate}
\end{lemma}

\begin{proof}  
Let $j:=\|DX\|+1$. Notice that for every  small enough $C^1$-perturbation $Z$ of $X$,
and any point $z\in M$, we have $\| P_Z^t(z)\| \leq e^{j\,t}$.
By lemma~\ref{max} we also have  
$\| P_Z^t(z)\|_\ast \leq (\sin^{-1}\theta)\,e^{j\,t}$,
where the norm $\|\cdot\|_\ast$ refers to the $P_X^\tau$-invariant decompositions
at the points $z$ and $X^t(z)$.
We set $K(\theta, m)= 4\,(\sin^{-2}\theta)\,e^{j\,m}$.
Once again we assume that our flow is linear by using Lemma~\ref{trivial}.
Take any hyperbolic periodic orbit $\gamma$ with period $\tau>m$.
Let $Y\in\mathfrak{X}_{\mu}^{\infty}(M)$ be the perturbation provided by
Lemma~\ref{interchange}, corresponding to the same $\epsilon$ and $\theta$
of this lemma. Without loss of generality, we may assume that
the point $q$ given in Lemma~\ref{interchange} is $q=p$.
Take matrix representations diagonalizing the linear 
Poincar\'e flow.
Notice that, by a) of Lemma~\ref{interchange}, 
this representation is conjugated to the original one 
by a conservative change of coordinates with norm bounded by $\sin^{-1}(\theta)$. 
Given $s\in [0,\tau-m]$ let $q=X^{-s}(p)$. Take $\phi(t)$, $\alpha(t)$ and $\sigma>1$ such that
$$P_{X}^{t}(q)=\begin{pmatrix}
\phi(t)  &  0 \\
0  &  \phi^{-1}(t) \\
\end{pmatrix},P_{X}^{\tau}(q)=\begin{pmatrix}
\sigma  &  0 \\
0  &  \sigma^{-1} \\
\end{pmatrix}\text{  and  } A(q,t)=\begin{pmatrix}
\alpha(t)  &  0 \\
0  &  -\alpha(t) \\
\end{pmatrix},$$
where $A(q,t)$ is the infinitesimal generator of $P_{X}^{t}(q)$, 
that is, 
$$ \frac{d}{dt} P_{X}^{t}(q)  = A(q,t)\cdot P_{X}^{t}(q)\;. $$ 
Clearly, $\phi(t)=e^{\int_{0}^{t}\alpha(r)dr}$\, 
and \,  $\phi(\tau)=e^{\int_{0}^{\tau}\alpha(r)dr}=\sigma$.
Notice also that for $t\in [0,\tau]$,\,
$\phi(t)=\| P_X^t(q)\|_\ast \leq (\sin^{-1}\theta)\,e^{j\,t}$.

We consider two cases:

If $\sigma \leq (\sin^{-1}\theta)\,e^{j\, m}$ consider
the $ P_X^\tau(q)$-invariant decomposition of $\R^2$. 
With respect to such splitting
 $ \|P_X^\tau(q)\|_\ast = \sigma$.
Then, applying  Lemma~\ref{max} \emph{1)}, we have
$$\|P^\tau_X(q)\| \leq 4\,(\sin^{-2}\theta)\,e^{j\,m} =  K\; , $$
 and there
is no need to perturb $X$. 

Otherwise, if  $\sigma > (\sin^{-1}\theta)\,e^{j\, m}$,
we define the continuous function
$$ \begin{array}{cccc}
\Theta(t)\colon & [0,\tau-m] & \longrightarrow & \mathbb{R} \\
& t & \longmapsto & \int_{0}^{t}\alpha(r)dr-\int_{t+m}^{\tau}\alpha(r)dr.
\end{array}$$
In this case, we have 
$$e^{\Theta(0)}= e^{ -\int_m^\tau \alpha } = 
 e^{\int_0^m \alpha  -\log\sigma} = \frac{\phi(m)}{\sigma}
 \leq \frac{(\sin^{-1}\theta)\,e^{j\,m}}{\sigma} <1\;,$$ 
which implies  $\Theta(0)< 0$.
We also have  
$$e^{\Theta(\tau-m)} = e^{\int_0^{\tau-m} \alpha} = 
 e^{\log\sigma- \int_{\tau-m}^\tau \alpha } =
\frac{\sigma}{e^{\int_{\tau-m}^\tau \alpha}}
 \geq \frac{\sigma}{(\sin^{-1}\theta)\,e^{j\,m}}  >1 \;,$$
which implies that $\Theta(\tau-m)>0$.
By the intermediate value theorem, there exists $s\in[0,\tau-m]$ such that 
\begin{equation}\label{defs}
\int_{0}^{s}\alpha(r)dr=\int_{s+m}^{\tau}\alpha(r)dr. 
\end{equation}
We fix such an $s\in[0,\tau-m]$ and $q=X^{-s}(p)$.
Since 
$$P_{X}^{\tau}(q)=P_{X}^{\tau-s-m}(X^{m}(p))\circ P_{X}^{m}(p)\circ P_{X}^{s}(q)$$ 
we have that $P_{X}^{\tau}(q)$ has the following matrix product representation 
$$
\begin{pmatrix}
e^{\int_{s+m}^{\tau}\alpha(r)dr}  &  0 \\
0  &  e^{-\int_{s+m}^{\tau}\alpha(r)dr}  \\
\end{pmatrix}\begin{pmatrix}
e^{\int_{s}^{s+m}\alpha(r)dr}   &  0 \\
0  &  e^{-\int_{s}^{s+m}\alpha(r)dr}  \\
\end{pmatrix}\begin{pmatrix}
\phi(s)  &  0 \\
0  &  \phi^{-1}(s) \\
\end{pmatrix}.
$$
Recall that $Y$, given by  Lemma~\ref{interchange},
is a conservative perturbation of $X$, supported in 
$\{X^{t}(p) \colon t\in[0,m]\}$, and such that $P_{Y}^{m}(N_{p}^{u})=N_{X^{m}(p)}^{s}$.
Therefore
$$P_{Y}^{m}(p)=\begin{pmatrix}
0   &  m_{1} \\
m_{2}  &  m_{3}  \\
\end{pmatrix},$$
for some constants $m_{i}$, $i=1,2,3$.
 A simple computation
(one just replaces the matrix above as the middle matrix
in the previous three-fold matrix composition) shows that
 $P_{Y}^{\tau}(q)=P_{X}^{\tau-s-m}(X^{m}(p))\circ P_{Y}^{m}(p)\circ P_{X}^{s}(q)$
is given by
$$P_{Y}^{\tau}(q)=
\begin{pmatrix}
0  &  m_{1}\phi^{-1}(s)e^{\int_{s+m}^{\tau}\alpha(r)dr} \\
m_{2}\phi(s)e^{-\int_{s+m}^{\tau}\alpha(r)dr}  &  m_{3}\phi^{-1}(s)e^{-\int_{s+m}^{\tau}\alpha(r)dr}  \\
\end{pmatrix} =
\begin{pmatrix}
0  &  m_{1} \\
m_{2}   &  m_{3}^\prime   \\
\end{pmatrix}\;,
$$ 
with $\vert m_3^\prime \vert = 
\vert m_3\vert \,e^{-\int_0^s \alpha -\int_{s+m}^{\tau}\alpha}\leq \vert m_3\vert $.
We have used~(\ref{defs}) here.
By Lemma~\ref{max} \emph{2)}
$$\max_{i=1,2,3} \vert m_{i}\vert = \|P_{Y}^{m}(p)\|_\ast \leq
\sin^{-1}(\theta)\|P_{Y}^{m}(p)\|
\leq \sin^{-1}(\theta)e^{jm}\;,$$
which implies that
$\|P_{Y}^{\tau}(q)\|_\ast \leq \max_{i=1,2,3} \vert m_{i}\vert
\leq \sin^{-1}(\theta)e^{jm}$. 
Finally, using Lemma~\ref{max} \emph{1)} we get 
$$\|P_{Y}^{\tau}(q)\|<4\sin^{-1}(\theta)\|P_{Y}^{\tau}(q)\|_{*}<4\sin^{-2}(\theta)e^{jm}=K,$$
and the lemma is proved.

\end{proof}

\begin{lemma}\label{perturb}
Let $X\in\mathfrak{X}_{\mu}^{\infty}(M)$, $\epsilon>0$ and $\theta>0$. 
Let $m=m(\epsilon/3,\theta)\in\mathbb{N}$ be given by Lemma~\ref{interchange}. 
There exists $T>m$ such that 
given any hyperbolic periodic point $q$ with period $\tau>T$ and satisfying a) and b) 
of Lemma~\ref{interchange}, then there exists a conservative $C^{1}$ perturbation 
$Z$ of $X$  such that 
\begin{enumerate}
\item $\|X-Z\|_{C^1} <\epsilon $;
\item $q$ is an elliptical point for $Z$, with the same period $\tau$.
\end{enumerate}
\end{lemma}

\begin{proof}
Let the data $X\in\mathfrak{X}_{\mu}^{\infty}(M)$, $\epsilon>0$ and $\theta>0$
be given.
Set $j:=\|DX\|+1$. Let $C>0$ be an upper bound for the norm
of the infinitesimal generator of the Poincar\'{e} linear flow
along any arc of regular orbit with length $\leq 1$.
The constant $C$ may be taken to be the product of $j$ with
the constant provided by Lemma~\ref{trivial} for $s=2$ and length one arcs.
Take $\delta:=\delta(C,\epsilon/3)$
given by Lemma~\ref{abstr:perturb}. 
Recall that $m:=m(\epsilon/3,\theta)\in\N$ was taken 
according to Lemma~\ref{interchange}.
Take $K:=K(\theta,m)$ according to Lemma~\ref{mainlemma}.
Consider now the following function.
For each $\theta>0$ and $\alpha>0$ let $\rho_\theta(\alpha)$
be the  Euclidean norm $\|e^Q-I\|$, for all matrices $Q\in sl(2,\R)$
with eigenvalues $\{\alpha,-\alpha\}$, and angle $\theta$
between its eigen-spaces. It is clear that all such matrices
are orthogonally conjugated, and, therefore, have the same
Euclidean norm. This function has the following properties:
\begin{enumerate}
\item Fixing $\theta>0$, the function $\alpha\mapsto \rho_\theta(\alpha)$
is a strictly increasing diffeomorphism from $]0,+\infty[$ onto $]0,+\infty[$.
\item Fixing $\alpha>0$, the function $\theta\mapsto \rho_\theta(\alpha)$
is a strictly decreasing diffeomorphism from $]0,\pi/2[$ onto $]\alpha,+\infty[$.
\end{enumerate}
Now, for $\delta$ and $\theta$ fixed above, set
$\alpha:=\alpha(\delta,\theta) = (\rho_\theta)^{-1}(\delta)$.
By definition, the number $\alpha>0$ has the following property:
Given any $\theta_1\geq \theta$, we can pick
 hyperbolic matrices $Q\in sl(2,\R)$ such that:
\begin{enumerate}
\item  $\|e^Q-I\|\leq\delta$;
\item $\alpha$ and $-\alpha$ are the eigenvalues of $Q$; and
\item $Q$ has an angle $\theta_1$ between its eigen-spaces.
\end{enumerate}
Finally, we define\,
$T:= \intpart{ \log K /\alpha } +2 $,
where $\intpart{\cdot}$ denotes the {\it integer part function}.
Now, let $\gamma = \{ X^t(p): 0\leq t\leq \tau \}$
be any hyperbolic periodic orbit, with period $\tau\geq T$,
satisfying a) and b) of Lemma~\ref{interchange}. 
Let  $Y\in\mathfrak{X}_{\mu}^{\infty}(M)$ be the vector field
provided by Lemma~\ref{mainlemma}, satisfying \emph{(1)} $\|X-Y\|_{C^{1}}<\epsilon/3$ and
\emph{(2)} $\|P_{Y}^{\tau}(q)\|<K$ for some point $q\in\gamma$. 
Because $\tau>T$, and $T$ is an integer, 
we can take $T$ disjoint time-one intervals $[t_i,t_i+1]$
inside the period $[0,\tau]$, i.e.,
$0\leq t_1<t_1+1 <t_2<t_2+1<\ldots <t_T<t_T+1 \leq \tau$.
Consider the corresponding points
$p_i := X^{t_i}(q)$ and $p_i^\prime:= X^1(p_i)=X^{t_i+1}(q)$.
Take the linear flow maps
$\Phi_i= P_Y^1(p_i):\R^2_{p_i}\to \R^2_{p_i^\prime}$
and let $\theta_i\geq\theta $ be the angle between the
eigen-spaces $N^u_{p_i}$ and $N^s_{p_i}$  of the $P_Y^\tau(p_i)$,
for each $i=1,\ldots, T$.
 Take now $Q_i\in sl(2,\R)$
such that
 $\|e^{Q_i}-I\|\leq\delta$, and
 $Q_i$ has eigen-space $N^u_{p_i}$ with eigenvalue $-\alpha$,
 and has eigen-space $N^s_{p_i}$ with eigenvalue $\alpha$.
Notice these eigen-spaces do make
an angle equal to $\theta_i$.
The product linear map
$\Phi_i\cdot e^{t\,Q_i}:\R^2_{p_i}\to \R^2_{p_i^\prime}$
takes the decomposition $\R^2_ {p_i}=N^u_{p_i}\oplus N^s_{p_i}$
onto the decomposition $\R^2_ {p_i^\prime}=N^u_{p_i^\prime}\oplus N^s_{p_i^\prime}$.
Moreover, we have
$\|\Phi_i\cdot e^{t\,Q_i}|_{N^{u}_{p_i}}\|=\|\Phi_i|_{N^{u}_{p_i}}\|\,e^{-\alpha\,t}$
and
$\|\Phi_i\cdot e^{t\,Q_i}|_{N^{s}_{p_i}}\|=\|\Phi_i|_{N^{s}_{p_i}}\|\,e^{\alpha\,t}$.
Feeding Lemma~\ref{abstr:perturb} with the
 input data $C$, $\epsilon/3$ and $S_i(t)= e^{t\, Q_i}$, $t\in [0,1]$,
and then combining it  with the pasting lemma~\ref{AM},
we get a conservative $C^1$ vector field $Z_i(t)$ such that
$\|Y-Z_i(t)\|\leq \epsilon/3$,\,
$Y=Z_i(t)$ outside a small neighbourhood of the arc $[p_i, p_i^\prime]\subset \gamma$,
and\, $P_{Z_i(t)}^1(p_i)= P_Y^1(p_i)\cdot e^{t\, Q_i}$.
But since all supports of these perturbations are disjoint, we can
glue them into a single  conservative $C^1$ perturbation $\tilde{Z}(t)$ of $X$
such that $\|Y-\tilde{Z}(t)\|\leq \epsilon/3$ and
$Y=\tilde{Z}(t)$ outside a small neighbourhood of $\gamma$.
By way of construction, the vector field $\tilde{Z}(t)$ has the same
invariant decomposition as $Y$. Moreover, we have
\begin{equation}\label{expn}
\varphi(t):=\|P_{\tilde{Z}(t)}^{\tau}(q)|_{N^{u}_{q}}\|= 
\|P_{Y}^{\tau}(q)|_{N^{u}_{q}}\|\,
 e^{-t\,T\,\alpha} <K \,  e^{-t\,T\,\alpha} \;,
\end{equation}
while, by conservative symmetry,
$\|P_{\tilde{Z}(t)}^{\tau}(q)|_{N^{s}_{q}}\|> K \,  e^{t\,T\,\alpha}$.
For $t=0$ we have $\varphi(0)=\|P_{\tilde{Z}(t)}^{\tau}(q)|_{N^{u}_{q}}\|>1$.
But since $T>\log K/\alpha$, we have  $K\,e^{-T\,\alpha}<1$,
and then for $t=1$,
\, $\varphi(1)<K\,e^{-T\,\alpha}<1$.
Therefore, there is some $t\in ]0,1[$ such that
$\varphi(t)=1$. For such $t_0$ we must have $P_{\tilde{Z}(t_0)}^\tau (q)=I$.

Finally, applying Lemma~\ref{smallang} to the periodic orbit
$\gamma$ of $\tilde{Z}(t_0)$, and to some $\xi>0$  (any will do),  
we get a conservative $C^1$ perturbation
$Z$ of $\tilde{Z}(t_0)$ such that
$\|Z-\tilde{Z}(t_0)\|\leq \epsilon/3$ and 
$\gamma$ is an elliptic periodic orbit of $Z$.
The perturbation $Z$ satisfies
$$\| X-Z\| \leq \underbrace{ \| X-Y\| }_{\leq \epsilon/3} + 
\underbrace{\| Y-\tilde{Z}(t_0)\| }_{\leq \epsilon/3}  + 
\underbrace{\| \tilde{Z}(t_0)-Z\| }_{\leq \epsilon/3} \leq 
\epsilon\;, $$
which completes the proof.
\end{proof}

\medskip

Finally by Lemmas~\ref{interchange}, ~\ref{mainlemma} and ~\ref{perturb} we obtain:

\begin{proposition}\label{propB}(Large angle perturbation)
Given $\epsilon>0$, $\theta>0$  and $X\in\mathfrak{X}_{\mu}^{1}(M)$, there exist \, 
$m\in \N$ \, and \,$T\in\N$\, (with $T>>m$) 
 such that given a  periodic hyperbolic point $p\in M$ with period $\tau>T$, 
 satisfying the conditions:
\begin{description}
 \item i) $\measuredangle(N_{q}^{u},N_{q}^{s})\geq \theta$ \, for all $q=X^{t}(p)$, $0\leq t\leq \tau$, and 
 \item ii) the orbit \, $X^t(p)$ \, has not $m$-dominated splitting, i.e.,
 $q=X^t(p)\in\Delta_m(X)$ for some $0\leq t\leq \tau$,
\end{description}
then there is some perturbation $Y\in\mathfrak{X}_{\mu}^{1}(M)$ of $X$ such that\,
 $\|X-Y\|_{C^1}<\epsilon$\, and 
$p$ is an elliptical periodic point for $Y^{t}$ with period $\tau$.
\end{proposition}

\medskip

\section{Proof of the main theorems}

There are several different ways of defining the \emph{topological dimension}
of topological space $X$, which we shall denote by $\mbox{dim}(X)$.
For separable metrizable spaces all these definitions are equivalent.
The topological dimension is, of course, a topological invariant.
Reference~\cite{HW} is a beautiful book on this subject.
We need here a theorem of Edward Szpilrajn
relating topological dimension
with Lebesgue measure. See the proof in~\cite{HW}.

\begin{theorem}[E. Szpilrajn]\label{dim_top}
Given $X\subset  \R^n$, if $X$ has zero Lebesgue measure then
$\mbox{dim}(X)<n$.
\end{theorem}

\medskip

\begin{proof}(of Theorem~\ref{intvazio}) 
Let $X$ be a class $C^1$ zero divergence vector field on
a compact boundaryless Riemannian manifold $M$.
Assume $X$ is not Anosov. Then, with a first small perturbation,
we can assume that $X$ is not in the boundary of Anosov vector fields
of $M$. By~\cite{Z}, we know that $\mathfrak{X}_{\mu}^{2}(M)$ is $C^{1}$-dense in
$\mathfrak{X}_{\mu}^{1}(M)$, where $\mu$ stands for the Riemannian volume of $M$.
Thus, given $\epsilon>0$, we can take a class $C^2$  vector field 
$X_1\in \mathfrak{X}_{\mu}^{2}(M)$ such that $\nrm{X-X_1}_{C^1}<\epsilon$
and $X_1$ is still not Anosov.
Taking $\epsilon$ small enough, the vector field $X_1$ has an invariant
hyperbolic set $\Lambda_1$ which is homeomorphically equivalent to $\Lambda$.
By the $3$-dimensional and continuous-time version of Theorem 14 of~\cite{BV3} (see~\cite{B2} and the references therein) it follows that $\Lambda_1$ has zero Lebesgue measure, and, applying Theorem~\ref{dim_top}, $\mbox{dim}(\Lambda_1)<3$.
Since $\Lambda$ and $\Lambda_1$ are homeomorphic, 
we derive that $\mbox{dim}(\Lambda)<3$.
But we would have $\mbox{dim}(\Lambda)=3$
in case $\Lambda$ has non-empty interior.
Therefore, $\Lambda$ has empty interior.
\end{proof}

\begin{remark}
Exactly the same argument can be used to prove
Proposition 1 of~\cite{X}, using Theorem 14 in~\cite{BV3}. Theorem~\ref{intvazio} should also be true for $n$-dimensional divergence-free vector fields by just adapting to the continuous-time setting the Theorem 14 of~\cite{BV3}. 
\end{remark}

\begin{proof}(of Theorem~\ref{main})
Let $\mathcal{KS}\subset \mathfrak{X}^{1}_{\mu}(M)$  be the residual set
given by Robinson's conservative version of Kupka-Smale theorem, see~\cite{R}.
Consider also the residual set $\mathcal{P}$ given by Pugh's general density theorem~\cite{PR}, say of divergence-free vector fields $X$ such that $\overline{Per(X)}=\Omega(X)=M$ (where the last equality follows by Poincar\'e recurrence).
Take any vector field $X\in \mathfrak{X}^{1}_{\mu}(M)$
which is not an Anosov. Through a first small perturbation we may, and will,
 assume that 
$X\in \mathcal{KS}\cap \mathcal{P}$ and $X$ is not Anosov.
Take  $p\in U\subseteq M$ where $U$ is an open set.
Now, given $\epsilon>0$ take $\theta=\theta(\epsilon,X)>0$
according to Proposition~\ref{propA}, and $m=m(\epsilon,\theta)\in\N$ and $T(m)$,
satisfying Proposition~\ref{propB}.
We consider four cases:
\begin{description}
\item a) Some elliptic closed orbit of $X$ goes through $U$.
\item b) All closed orbits of $X$ which go through $U$ are hyperbolic, and
some of them has a small angle, less than $\theta$, between stable and unstable
directions.
\item c) All closed orbits of $X$ which go through $U$ are hyperbolic, with
angle between stable and unstable directions bounded from bellow by $\theta$,
but  some of them has not $m$-dominated splitting, i.e., it goes through
$\Delta_m(X)$.
\item d) All closed orbits of $X$ which go through $U$ are hyperbolic, with
$m$-dominated splitting, and with an
angle between stable and unstable directions bounded from bellow by $\theta$. 
\end{description}

In case a) there is nothing to prove, just let $Y=X$.
Case b) follows from Proposition~\ref{propA}. 
Analogously,  case c) follows from Proposition~\ref{propB}. 
Finally, we prove that case d) is impossible.
Let $\Lambda$ be the closure of all closed orbits of $X$ which go through $U$.
From d) it  follows that $\Lambda$ has $m$-dominated splitting.
The dimension of $M$, plus the volume-preserving and the absence of singularities assumptions, imply that
$\Lambda$ is hyperbolic for the linear Poincar\'e flow, thus hyperbolic for the flow (see Remark~\ref{Doering}).
Because $X\in\mathcal{P}$, we must have $U\subseteq \Lambda$.
By Theorem~\ref{intvazio}, then $X$ must be Anosov, which contradicts
the assumption on $X$.
\end{proof}

\medskip

Let $\mathcal{N}$ be the complement of the $C^1$-closure of Anosov systems
in $\mathfrak{X}^{1}_{\mu}(M)$, and  $\Phi$ be the subset of  $\mathfrak{X}^{1}_{\mu}(M)\!\times\!M\!\times \R_+$
of all triples  $(X,x,\epsilon)$ such that
$X$ has a closed elliptic orbit going through the ball $B(x,\epsilon)$.
This set is open, because elliptic orbits are stable.

Given any open set  $\mathcal{U}\subseteq \mathcal{N}$ 
consider the following (also open) set 
$$\Phi(\mathcal{U},x,\epsilon)=\{ \,Y\in \mathcal{U}\,:\; (Y,x,\epsilon)\in \Phi\;\}\;.$$

\medskip

\begin{lemma}
Given $\epsilon>0$, $p\in M$ and an open set $\mathcal{U}\subseteq\mathcal{N}$,
$\Phi(\mathcal{U},p,\epsilon)$ is an open dense subset of $\mathcal{U}$.
\end{lemma}

\begin{proof}
Follows by Theorem~\ref{main}.
\end{proof}

\begin{proof}(of Theorem~\ref{main_theor})
Let $(x_n)$ be a sequence dense in $M$, and $(\epsilon_n)$ a
sequence of positive real numbers such that
$\lim_{n\rightarrow\infty} \epsilon_n=0$.
Defining recursively
$$ \mathcal{U}_0 = \mathcal{N}\qquad \text{ and } \qquad 
\mathcal{U}_{n+1}=\Phi(\mathcal{U}_n, x_n,\epsilon_n)\quad (n\geq 1)$$
the residual set
$\mathcal{R}=\cap_{n=1}^\infty \mathcal{U}_n$ 
is such that  for all $Y\in\mathcal{R}$,
the elliptic closed orbits of $Y$ are dense in $M$.
Denote by $\mathcal{A}$ the subset of all Anosov systems in $\mathfrak{X}^{1}_{\mu}(M)$.
Then $\tilde{\mathcal{R}} = \mathcal{A}\cup\mathcal{R}$ is a residual
subset of $\mathfrak{X}^{1}_{\mu}(M)$, for which the dichotomy of Theorem~\ref{main_theor}
holds.
\end{proof}

%For acknowledgements section, please don't number the section, you need to begin with \section*{Acknowledgements}
%\section*{Acknowledgements} We would like to thank the
%referees very much for their valuable comments and suggestions.

\medskip

\flushleft
{\bf M\'ario Bessa} \ \  (bessa@fc.up.pt)\\
CMUP, Rua do Campo Alegre, 687 \\ 4169-007 Porto \\ Portugal\\

\flushleft
{\bf Pedro Duarte} \ \  (pduarte@ptmat.fc.ul.pt)\\
CMAF, Faculdade de Ci\^{e}ncias da Universidade de Lisboa, Campo Grande, Edificio C6 \\ 1749-016 Lisboa \\ Portugal\\

\end{document}